\documentclass[12pt]{amsart}

\usepackage{graphics}
\usepackage{amsmath}
\usepackage{amsfonts}
\usepackage{amssymb}
\usepackage{amscd}
\usepackage{mathrsfs}
\input xy
\xyoption{all}

\newcommand{\C}{{\mathbb C}}

\renewcommand{\O}{{\mathscr O}}
\newcommand{\R}{{\mathbb R}}

\newtheorem{theorem}{\bf Theorem}
\newtheorem{lemma}{\bf Lemma}

\title{Applications of a parametric Oka principle for liftings}

\author{Finnur L\'arusson}
\address{School of Mathematical Sciences, University of Adelaide, Adelaide SA 5005, Australia.} 
\email{finnur.larusson@adelaide.edu.au}

\dedicatory{Dedicated to Linda P.\ Rothschild}

\subjclass[2000]{Primary 32Q28.  Secondary 18F20, 18G55, 32E10, 55U35.}

\keywords{Stein manifold, Oka principle, Oka property, convex interpolation property, sub\-elliptic, model structure, fibration.}

\date{28 January 2009.  Minor changes 4 May 2009}

\begin{document}

\begin{abstract}  
A parametric Oka principle for liftings, recently proved by Forstneri\v c, provides many examples of holomorphic maps that are fibrations in a model structure introduced in previous work of the author.  We use this to show that the basic Oka property is equivalent to the parametric Oka property for a large class of manifolds.  We introduce new versions of the basic and parametric Oka properties and show, for example, that a complex manifold $X$ has the basic Oka property if and only if every holomorphic map to $X$ from a contractible submanifold of $\C^n$ extends holomorphically to $\C^n$.
\end{abstract}

\maketitle

\section{Introduction}

\noindent
In this note, which is really an addendum to the author's papers \cite{Larusson1} and \cite{Larusson2}, we use a new parametric Oka principle for liftings, very recently proved by Franc Forstneri\v c \cite{Forstneric0}, to solve several problems left open in those papers.  To make this note self-contained would require a large amount of technical background.  Instead, we give a brief introduction and refer to the papers \cite{Forstneric1}, \cite{Forstneric2}, \cite{Larusson1}, \cite{Larusson2} for more details.  

The modern theory of the Oka principle began with Gromov's seminal paper \cite{Gromov} of 1989.  Since then, researchers in Oka theory have studied more than a dozen so-called Oka properties that a complex manifold $X$ may or may not have.  These properties concern the task of deforming a continuous map $f$ from a Stein manifold $S$ to $X$ into a holomorphic map.  If this can always be done so that under the deformation $f$ is kept fixed on a closed complex submanifold $T$ of $S$ on which $f$ is holomorphic, then $X$ is said to have the basic Oka property with interpolation (BOPI).  Equivalently (this is not obvious), $T$ may be taken to be a closed analytic subvariety of a reduced Stein space $S$.  If $f$ can always be deformed to a holomorphic map so that the deformed maps stay arbitrarily close to $f$ on a holomorphically convex compact subset $K$ of $S$ on which $f$ is holomorphic and are holomorphic on a common neighbourhood of $K$, then $X$ is said to have the basic Oka property with approximation (BOPA).  If every holomorphic map to $X$ from a compact convex subset $K$ of $\C^n$ can be approximated uniformly on $K$ by entire maps $\C^n\to X$, then $X$ is said to have the convex approximation property (CAP), introduced in \cite{Forstneric2}.  These properties all have parametric versions (POPI, POPA, PCAP) where instead of a single map $f$ we have a family of maps depending continuously on a parameter.  

Some of the basic and parametric Oka properties have been extended from complex manifolds to holomorphic maps (viewing a manifold as a constant map from itself).  For example, a holomorphic map $f:X\to Y$ is said to satisfy POPI if for every Stein inclusion $T\hookrightarrow S$ (a Stein manifold $S$ with a submanifold $T$), every finite polyhedron $P$ with a subpolyhedron $Q$, and every continuous map $g:S\times P\to X$ such that the restriction $g|S\times Q$ is holomorphic along $S$, the restriction $g|T\times P$ is holomorphic along $T$, and the composition $f\circ g$ is holomorphic along $S$, there is a continuous map $G:S\times P\times I\to X$, where $I=[0,1]$, such that:
\begin{enumerate}
\item  $G(\cdot,\cdot,0)=g$,
\item  $G(\cdot,\cdot,1):S\times P\to X$ is holomorphic along $S$,
\item  $G(\cdot,\cdot,t)=g$ on $S\times Q$ and on $T\times P$ for all $t\in I$,
\item  $f\circ G(\cdot,\cdot,t)=f\circ g$ on $S\times P$ for all $t\in I$.
\end{enumerate}
Equivalently, $Q\hookrightarrow P$ may be taken to be any cofibration between cofibrant topological spaces, such as the inclusion of a subcomplex in a CW-complex, and the existence of $G$ can be replaced by the stronger statement that the inclusion into the space, with the compact-open topology, of continuous maps $h:S\times P\to X$ with $h=g$ on $S\times Q$ and on $T\times P$ and $f\circ h=f\circ g$ on $S\times P$ of the subspace of maps that are holomorphic along $S$ is acyclic, that is, a weak homotopy equivalence (see \cite{Larusson1}, \S 16).  (Here, the notion of cofibrancy for topological spaces and continuous maps is the stronger one that goes with Serre fibrations rather than Hurewicz fibrations.  We remind the reader that a Serre fibration between smooth manifolds is a Hurewicz fibration, so we will simply call such a map a topological fibration.)

In \cite{Larusson1}, the category of complex manifolds was embedded into the category of prestacks on a certain simplicial site with a certain simplicial model structure such that all Stein inclusions are cofibrations, and a holomorphic map is acyclic if and only if it is topologically acyclic, and is a fibration if and only if it is a topological fibration and satisfies POPI.  It was known then that complex manifolds with the geometric property of subellipticity satisfy POPI, but very few examples of nonconstant holomorphic maps satisfying POPI were known, leaving some doubt as to whether the model structure constructed in \cite{Larusson1} is an appropriate homotopy-theoretic framework for the Oka principle.  This doubt is dispelled by Forstneri\v c's parametric Oka principle for liftings.

Forstneri\v c has proved that the basic Oka properties for manifolds are equivalent, and that the parametric Oka properties for manifolds are equivalent (\cite{Forstneric1}, see also \cite{Larusson2}, Theorem 1; this is not yet known for maps), so we can refer to them as the basic Oka property and the parametric Oka property, respectively.  For Stein manifolds, the basic Oka property is equivalent to the parametric Oka property  (\cite{Larusson2}, Theorem 2).  In \cite{Larusson2} (the comment following Theorem 5), it was noted that the equivalence of all the Oka properties could be extended to a much larger class of manifolds, including for example all quasi-projective manifolds, if we had enough examples of holomorphic maps satisfying POPI.  This idea is carried out below.  It remains an open problem whether the basic Oka property is equivalent to the parametric Oka property for all manifolds.

We conclude by introducing a new Oka property that we call the convex interpolation property, with a basic version equivalent to the basic Oka property and a parametric version equivalent to the parametric Oka property.  In particular, we show that a complex manifold $X$ has the basic Oka property if and only if every holomorphic map to $X$ from a contractible submanifold of $\C^n$ extends holomorphically to $\C^n$.  This is based on the proof of Theorem 1 in \cite{Larusson2}.

{\it Acknowledgement.}  I am grateful to Franc Forstneri\v c for helpful discussions.

\section{The parametric Oka principle for liftings}

\noindent
Using the above definition of POPI for holomorphic maps, we can state Forstneri\v c's parametric Oka principle for liftings, in somewhat less than its full strength, as follows.

\begin{theorem}[Parametric Oka principle for liftings \cite{Forstneric0}]  Let $X$ and $Y$ be complex manifolds and $f:X\to Y$ be a holomorphic map which is either a subelliptic submersion or a holomorphic fibre bundle whose fibre has the parametric Oka property.  Then $f$ has the parametric Oka property with interpolation.
\label{POPL}
\end{theorem}

The notion of a holomorphic submersion being subelliptic was introduced by Forstneri\v c \cite{Forstneric3}, generalising the concept of ellipticity due to Gromov \cite{Gromov}.  Subellipticity is the weakest currently-known sufficient geometric condition for a holomorphic map to satisfy POPI.

By a corollary of the main result of \cite{Larusson1}, Corollary 20, a holomorphic map $f$ is a fibration in the so-called intermediate model structure constructed in \cite{Larusson1} if and only if $f$ is a topological fibration and satisfies POPI (and then $f$ is a submersion).  In particular, considering the case when $f$ is constant, a complex manifold is fibrant if and only if it has the parametric Oka property.  The following result is therefore immediate.

\begin{theorem}  {\rm (1)}  A subelliptic submersion is an intermediate fibration if and only if it is a topological fibration. 
 
{\rm (2)}  A holomorphic fibre bundle is an intermediate fibration if and only if its fibre has the parametric Oka property.
\label{conj-solved}
\end{theorem}

Part (1) is a positive solution to Conjecture 21 in \cite{Larusson1}.  For the only-if direction of (2), we simply take the pullback of the bundle by a map from a point into the base of the bundle and use the fact that in any model category, a pullback of a fibration is a fibration.  As remarked in \cite{Larusson1}, (1) may be viewed as a new manifestation of the Oka principle, saying that for holomorphic maps satisfying the geometric condition of subellipticity, there is only a topological obstruction to being a fibration in the holomorphic sense defined by the model structure in \cite{Larusson1}.  Theorem \ref{conj-solved} provides an ample supply of intermediate fibrations.

A result similar to our next theorem appears in \cite{Forstneric0}.  The analogous result for the basic Oka property is Theorem 3 in \cite{Larusson2}.

\begin{theorem}  Let $X$ and $Y$ be complex manifolds and $f:X\to Y$ be a holomorphic map which is an intermediate fibration.  

{\rm (1)}  If $Y$ satisfies the parametric Oka property, then so does $X$.  

{\rm (2)}  If $f$ is acyclic and $X$ satisfies the parametric Oka property, then so does $Y$.
\label{up-and-down}
\end{theorem}

\begin{proof}  (1)  This follows immediately from the fact that if the target of a fibration in a model category is fibrant, so is the source.

(2)   Let $T\hookrightarrow S$ be a Stein inclusion and $Q\hookrightarrow P$ an inclusion of parameter spaces.  Let $h:S\times P\to Y$ be a continuous map such that the restriction $h|S\times Q$ is holomorphic along $S$ and the restriction $h|T\times P$ is holomorphic along $T$.  We need a lifting $k$ of $h$ by $f$ with the same properties.  Then the parametric Oka property of $X$ allows us to deform $k$ to a continuous map $S\times P\to X$ which is holomorphic along $S$, keeping the restrictions to $S\times Q$ and $T\times P$ fixed.  Finally, we compose this deformation by $f$.

To obtain the lifting $k$, we first note that since $f$ is an acyclic topological fibration, $h|T\times Q$ has a continuous lifting, which, since $f$ satisfies POPI, may be deformed to a lifting which is holomorphic along $T$.  We use the topological cofibration $T\times Q\hookrightarrow S\times Q$ to extend this lifting to a continuous lifting $S\times Q\to X$, which may be deformed to a lifting which is holomorphic along $S$, keeping the restriction to $T\times Q$ fixed.  We do the same with $S\times Q$ replaced by $T\times P$ and get a continuous lifting of $h$ restricted to $(S\times Q)\cup(T\times P)$ which is holomorphic along $S$ on $S\times Q$ and along $T$ on $T\times P$.  Finally, we obtain $k$ as a continuous extension of this lifting, using the topological cofibration $(S\times Q)\cup (T\times P)\hookrightarrow S\times P$.
\end{proof}

\section{Equivalence of the basic and the parametric Oka properties}

\noindent
Following \cite{Larusson2}, by a {\it good}\, map we mean a holomorphic map which is an acyclic intermediate fibration, that is, a topological acyclic fibration satisfying POPI.  We call a complex manifold $X$ {\it good}\, if it is the target, and hence the image, of a good map from a Stein manifold.  This map is then weakly universal in the sense that every holomorphic map from a Stein manifold to $X$ factors through it.

A Stein manifold is obviously good.  As noted in \cite{Larusson2}, the class of good manifolds is closed under taking submanifolds, products, covering spaces, finite branched covering spaces, and complements of analytic hypersurfaces.  This does not take us beyond the class of Stein manifolds.  However, complex projective space, and therefore every quasi-projective manifold, carries a holomorphic affine bundle whose total space is Stein (in algebraic geometry this observation is called the Jouanolou trick), and by Theorem \ref{conj-solved}, the bundle map is good.  Therefore all quasi-projective manifolds are good.  (A quasi-projective manifold is a complex manifold of the form $Y\setminus Z$, where $Y$ is a projective variety and $Z$ is a subvariety.  We need the fact, proved using blow-ups, that $Y$ can be taken to be smooth and $Z$ to be a hypersurface.)  The class of good manifolds thus appears to be quite large, but we do not know whether every manifold, or even every domain in $\C^n$, is good.

\begin{theorem}   A good manifold has the basic Oka property if and only if it has the parametric Oka property.
\label{equivalence}
\end{theorem}

\begin{proof}  Let $S\to X$ be a good map from a Stein manifold $S$ to a complex manifold $X$.  If $X$ has the basic Oka property, then so does $S$ by \cite{Larusson2}, Theorem 3.  Since $S$ is Stein, $S$ is elliptic by \cite{Larusson2}, Theorem 2, so $S$ has the parametric Oka property.  By Theorem \ref{up-and-down}, it follows that $X$ has the parametric Oka property.
\end{proof}

\section{The convex interpolation property}

\noindent
Let us call a submanifold $T$ of $\C^n$ {\it special}\, if $T$ is the graph of a proper holomorphic embedding of a convex domain $\Omega$ in $\C^k$, $k\geq 1$, as a submanifold of $\C^{n-k}$, that is, 
$$T=\{(x,\varphi(x))\in\C^k\times\C^{n-k}:x\in\Omega\},$$ 
where $\varphi:\Omega\to\C^{n-k}$ is a proper holomorphic embedding.  We say that a complex manifold $X$ satisfies the {\it convex interpolation property} (CIP) if every holomorphic map to $X$ from a special submanifold $T$ of $\C^n$ extends holomorphically to $\C^n$, that is, the restriction map $\O(\C^n,X)\to\O(T,X)$ is surjective.

We say that $X$ satisfies the {\it parametric convex interpolation property} (PCIP) if whenever $T$ is a special submanifold of $\C^n$, the restriction map $\O(\C^n,X)\to\O(T,X)$ is an acyclic Serre fibration.  (Since $\C^n$ and $T$ are holomorphically contractible, acyclicity is automatic; it is the fibration property that is at issue.)  More explicitly, $X$ satisfies PCIP if whenever $T$ is a special submanifold of $\C^n$ and $Q\hookrightarrow P$ is an inclusion of parameter spaces, every continuous map $f:(\C^n\times Q)\cup(T\times P)\to X$, such that $f|\C^n\times Q$ is holomorphic along $\C^n$ and $f|T\times P\to X$ is holomorphic along $T$, extends to a continuous map $g:\C^n\times P\to X$ which is holomorphic along $\C^n$.  The parameter space inclusions $Q\hookrightarrow P$ may range over all cofibrations of topological spaces or, equivalently, over the generating cofibrations $\partial B_n\hookrightarrow B_n$, $n\geq 0$, where $B_n$ is the closed unit ball in $\R^n$ (we take $B_0$ to be a point and $\partial B_0$ to be empty).  Clearly, CIP is PCIP with $P$ a point and $Q$ empty.  

\begin{lemma}  A complex manifold has the parametric convex interpolation property if and only if it has the parametric Oka property with interpolation for every inclusion of a special submanifold into $\C^n$.
\label{lemma}
\end{lemma}

\begin{proof}  Using the topological acyclic cofibration $(\C^n\times Q)\cup(T\times P)\hookrightarrow \C^n\times P$, we can extend a continuous map $f:(\C^n\times Q)\cup(T\times P)\to X$ as in the definition of PCIP to a continuous map $g:\C^n\times P\to X$.  POPI allows us to deform $g$ to a continuous map $h:\C^n\times P\to X$ which is holomorphic along $\C^n$, keeping the restriction to $(\C^n\times Q)\cup(T\times P)$ fixed, so $h$ extends $f$.

Conversely, if $h:\C^n\times P\to X$ is a continuous map such that $h|\C^n\times Q$ is holomorphic along $\C^n$ and $h|T\times P$ is holomorphic along $T$, and $g$ is an extension of $f=h|(\C^n\times Q)\cup(T\times P)$ provided by PCIP, then the topological acyclic cofibration
$$\big(\big((\C^n\times Q)\cup(T\times P)\big)\times I\big)\cup\big(\C^n\times P\times\{0,1\}\big)\hookrightarrow \C^n\times P\times I$$
provides a deformation of $h$ to $g$ which is constant on $(\C^n\times Q)\cup(T\times P)$. 
\end{proof}

\begin{theorem}  A complex manifold has the convex interpolation property if and only if it has the basic Oka property.  A complex manifold has the parametric interpolation property if and only if it has the parametric Oka property.
\label{cep}
\end{theorem}

\begin{proof}  We prove the equivalence of the parametric properties.  The equivalence of the basic properties can be obtained by restricting the argument to the case when $P$ is a point and $Q$ is empty.  By Lemma \ref{lemma}, POPI, which is one of the equivalent forms of the parametric Oka property by \cite{Forstneric1}, Theorem 6.1, implies PCIP.

By \cite{Larusson2}, Theorem 1, POPI implies POPA (not only for manifolds but also for maps).  The old version of POPA used in \cite{Larusson2} does not require the intermediate maps to be holomorphic on a neighbourhood of the holomorphically convex compact subset in question, only arbitrarily close to the initial map.  This property is easily seen to be equivalent to the current, ostensibly stronger version of POPA: see the comment preceding Lemma 5.1 in \cite{Forstneric1}.

The proof of Theorem 1 in \cite{Larusson2} shows that to prove POPA for $K$ convex in $S=\C^k$, that is, to prove PCAP, it suffices to have POPI for Stein inclusions $T\hookrightarrow \C^n$ associated to convex domains $\Omega$ in $\C^k$ as in the definition of a special submanifold.  Thus, by Lemma \ref{lemma}, PCIP implies PCAP, which is one of the equivalent forms of the parametric Oka property, again by \cite{Forstneric1}, Theorem 6.1.
\end{proof}

There are many alternative definitions of a submanifold of $\C^n$ being special for which Theorem \ref{cep} still holds.  For example, we could define special to mean topologically contractible: this is the weakest definition that obviously works.  We could also define a submanifold of $\C^n$ to be special if it is biholomorphic to a bounded convex domain in $\C^k$, $k<n$.  On the other hand, for the proof of Theorem \ref{cep} to go through, the class of special manifolds must contain $T$ associated as above to every element $\Omega$ in some basis of convex open neighbourhoods of every convex compact subset $K$ of $\C^n$ for every $n\geq 1$, such that $K$ is of the kind termed {\it special}\, by Forstneri\v c (see \cite{Forstneric2}, Section 1).

\end{document}